\documentclass[12pt]{amsart}

\usepackage{amssymb,amsmath,amsthm}

\def\RR{{\mathbb R}}
\def\ZZ{{\mathbb Z}}

\newtheorem{Theorem}{Theorem}[section]
\newtheorem{Lemma}[Theorem]{Lemma}
\newtheorem{Corollary}[Theorem]{Corollary}
\newtheorem{Proposition}[Theorem]{Proposition}
\theoremstyle{definition}
\newtheorem{Remark}[Theorem]{Remark}
\newtheorem{Example}[Theorem]{Example}

\begin{document}

\title{Centrally symmetric configurations of\\
order polytopes}

\author{Takayuki Hibi}
\address[Takayuki Hibi]{
Department of Pure and Applied Mathematics,
Graduate School of Information Science and Technology,
Osaka University,
Toyonaka, Osaka 560-0043, Japan}
\email{hibi@math.sci.osaka-u.ac.jp}

\author{Kazunori Matsuda}
\address[Kazunori Matsuda]{
Department of Pure and Applied Mathematics,
Graduate School of Information Science and Technology,
Osaka University,
Toyonaka, Osaka 560-0043, Japan}
\email{kaz-matsuda@math.sci.osaka-u.ac.jp}

\author{Hidefumi Ohsugi}
\address[Hidefumi Ohsugi]{Department of 
Mathematical Sciences,
School of Science and Technology,
Kwansei Gakuin University,
Sanda, Hyogo, 669-1337, Japan}
\email{ohsugi@kwansei.ac.jp}

\author{Kazuki Shibata}
\address[Kazuki Shibata]{
Department of Mathematics,
College of Science,
Rikkyo University,
Toshima-ku, Tokyo 171-8501, Japan.}
\email{k-shibata@rikkyo.ac.jp}

\begin{abstract}
It is shown that the toric ideal of the centrally symmetric
configuration of the order polytope of a finite partially ordered set
possesses a squarefree quadratic initial ideal.  It then follows that
the convex polytope arising from the centrally symmetric configuration
of an order polytope is a normal Gorenstein Fano polytope.
\end{abstract}

\maketitle

\section*{Introduction}
Gorenstein Fano polytopes (reflexive polytopes)
are interested in many researchers
since they correspond to Gorenstein Fano varieties
and are related with mirror symmetry.
See, e.g., \cite[\S 8.3]{Cox} and its References.
One of the most important problem 
is to find new classes of Gorenstein Fano polytopes.
The centrally symmetric configuration \cite{central} of an integer matrix 
supplies one of the most powerful tools to construct
normal Gorenstein Fano polytopes.
The purpose of the present paper is to study the centrally symmetric configuration 
of the integer matrix associated with the order polytope of 
a finite partially ordered set.

Let $\mathbb{Z}^{d \times n}$ denote the set of $d \times n$ integer matrices. 
Given $A \in \mathbb{Z}^{d \times n}$ for which no column vector is a zero vector,
the {\em centrally symmetric configuration}
of $A$ is the $(d+1) \times (2n + 1)$ integer matrix 
\[
A^{\pm} = 
\left[
\begin{array}{c|ccc|ccc}
0      &   &   &    & &   & \\
\vdots &   & A &    & &-A & \\
0      &   &   &    & &   & \\
\hline
1      & 1 & \cdots & 1 & 1 & \cdots & 1
\end{array}
\right].
\]
On the other hand, 
the {\em centrally symmetric polytope} 
arising from $A$ is the convex polytope ${\mathcal Q}^{\rm (sym)}_{A}$ 
which is the convex hull in $\RR^{d}$ of the column vectors of the matrix
\[
\left[
\begin{array}{c|ccc|ccc}
0      &   &   &    & &   & \\
\vdots &   & A &    & &-A & \\
0      &   &   &    & &   & 
\end{array}
\right].
\] 

We focus our attention on the problem when ${\mathcal Q}^{\rm (sym)}_{A}$ is 
a normal Gorenstein Fano polytope.
In general, 
the origin is contained in the interior of
${\mathcal Q}^{\rm (sym)}_{A} \subset \RR^{d}$.
Suppose that 
$A \in \ZZ^{d \times n}$ satisfies $\ZZ A = \ZZ^d$.
(Here $\ZZ A = \{ z_1 {\bf a}_1 + \cdots + z_n {\bf a}_n
\ | \  z_1 , \ldots, z_n \in \ZZ \}$
for $A = [{\bf a}_1, \ldots, {\bf a}_n] \in \ZZ^{d \times n}$.)
Then Lemma \ref{Gorenstein} in Section 1 guarantees that
if the toric ideal $I_{A^{\pm}}$ of $A^{\pm}$ possesses a squarefree initial
ideal with respect to a reverse lexicographic order for which the variable
corresponding to the column $[0, \ldots, 0, 1]^{t}$ is smallest,
then ${\mathcal Q}^{\rm (sym)}_{A}$ is a normal Gorenstein Fano polytope.

In \cite{central}, it is shown that if 
${\rm rank}(A) = d$ and 
all nonzero maximal minors of $A$ are $\pm 1$,
then ${\mathcal Q}^{\rm (sym)}_{A}$ is a normal Gorenstein Fano polytope. 
However, the converse does not hold. 
It is mentioned the existence of a matrix $A$ such that 
$A$ does not satisfy the above condition but ${\mathcal Q}^{\rm (sym)}_{A}$ is a 
normal Gorenstein Fano polytope (\cite[Example 2.16]{central}). 
Hence to construct an infinite family of matrices $A$ like this is an important problem. 
The aim of this paper is to give an answer of this problem by using partially ordered sets. 

Let $P$ be a partially ordered set (poset) 
on $[d] = \{ 1, \ldots, d \}$
and ${\bf e}_{1}, \ldots, {\bf e}_{d}$ the unit coordinate vectors 
of $\RR^{d}$.
Given a subset $\alpha \subset P$, we write $\rho(\alpha) \in \RR^{d}$ 
for the vector $\sum_{i \in \alpha}{\bf e}_{i}$. 
A {\em poset ideal} of $P$ is a subset $\alpha \subset P$ such that
if $a \in \alpha$ and $b \in P$ together with $b \leq a$, 
then $b \in \alpha$.  In particular, the empty set as well as $P$
itself is a poset ideal of $P$.  Let ${\mathcal J}(P)$ denote the set
of poset ideals of $P$.   
Note that $\mathcal{J}(P)$ has the structure of distributive lattice under inclusion, 
moreover, for any distributive lattice $D$, there exists a poset $P$ 
such that $D \cong \mathcal{J}(P)$ by Birkhoff's Theorem \cite{Birkhoff}. 
The {\em order polytope} ${\mathcal O}(P)$ is the 
$d$-dimensional polytope
which is the convex hull of
$\{ \rho(\alpha) \, | \, \alpha \in {\mathcal J}(P) \}$ in $\RR^{d}$.
See \cite{Stanley}.
We then write $A_{P}$ for the integer matrix whose column vectors are
those $\rho(\alpha)^{t}$ with $\alpha \in {\mathcal J}(P) \setminus
\{ \emptyset \}$.  
This integer matrix $A_P$ always satisfies $\ZZ A_P = \ZZ^d$, but 
does not always satisfy the condition that 
all nonzero maximal minors of $A_P$ are $\pm 1$ (Proposition 2.4). 
We prove that ${\mathcal Q}^{\rm (sym)}_{A_{P}}$ is a normal Gorenstein Fano polytope 
for {\em any} partially ordered set $P$. 

 

The present paper is organized as follows.  In Section $1$, we recall
basic materials on toric ideals of configurations as well as Fano polytopes.
In Section $2$, we will prove that,
for an arbitrary finite partially ordered set $P$,
the toric ideal $I_{A_{P}^{\pm}}$ of $A_{P}^{\pm}$ 
possesses a squarefree quadratic initial ideal with respect to 
a reverse lexicographic order for which the variable
corresponding to the column $[0, \ldots, 0, 1]^{t}$ is smallest
(Theorem \ref{Grobner}).
In particular, 
the centrally symmetric polytope ${\mathcal Q}^{\rm (sym)}_{A_{P}}$
arising from an arbitrary finite partially ordered set $P$
is a normal Gorenstein Fano polytope
(Corollary \ref{NGF}).
%
Finally, in Section $3$, we compute
the $\delta$-vectors of the centrally symmetric polytope 
${\mathcal Q}^{\rm (sym)}_{A_{P}}$, where $P$ is an antichain.

\section{Centrally symmetric configurations and Fano polytopes}

We recall fundamental materials on centrally symmetric configurations and 
Fano polytopes.  Let, as before, $\mathbb{Z}^{d \times n}$ denote the set of
$d \times n$ integer matrices.

\medskip

\noindent
{\bf a) Configuration}

A matrix $A \in \mathbb{Z}^{d \times n}$ is called a {\em configuration} 
if there exists a hyperplane $\mathcal{H} \subset \mathbb{R}^d$ not passing the origin 
of $\mathbb{R}^d$ such that each column vector of $A$ lies on $\mathcal{H}$. 

Let $K$ be a field and $K[T^{\pm 1}] = K[t_1, t_1^{-1}, \ldots, t_d, t_d^{-1}]$ 
the Laurent polynomial ring in $d$ variables over $K$. 
We associate each vector ${\bf a} = [a_1, \ldots, a_d]^t \in \mathbb{Z}^d$, 
with the Laurent monomial
$T^{\bf{a}} = t_1^{a_1} \cdots t_d^{a_d} \in K[T^{\pm 1}]$.
Given a configuration $A \in \mathbb{Z}^{d \times n}$ with 
$\mathbf{a}_1, \ldots, \mathbf{a}_n$ its column vectors, 
the {\em toric ring} $K[A]$ of $A$ is the monomial subalgebra of $K[T^{\pm 1}]$ 
which is generated by $T^{{\bf a}_1}, \ldots, T^{{\bf a}_n}$. 
Let $K[X] = K[x_1, \ldots, x_n]$ denote the polynomial ring in $n$ variables 
over $K$.  We define the 
surjective ring homomorphism $\pi : K[X] \to K[A]$ 
by setting $\pi(x_i) = T^{{\bf a}_i}$ for $i = 1, \ldots, n$.
The kernel $I_A$ of $\pi$ is called
the {\em toric ideal} of $A$. 

\medskip

\noindent
{\bf b) Fano polytope}

A convex polytope ${\mathcal P} \subset \RR^{d}$ is called {\em integral}
if each vertex of ${\mathcal P}$ belongs to $\ZZ^{d}$.  We say that
an integral convex polytope ${\mathcal P} \subset \RR^{d}$ is {\em normal}
if, for each integer $N > 0$ and for each ${\bf a} \in N{\mathcal P} \cap \ZZ^{d}$,
there exist ${\bf a}_{1}, \ldots, {\bf a}_{N}$ belonging to 
${\mathcal P} \cap \ZZ^{d}$ such that 
${\bf a} = {\bf a}_{1} + \cdots + {\bf a}_{N}$. 
Now, an integral convex polytope ${\mathcal P} \subset \RR^d$ 
is said to be a {\em Fano polytope} if the dimension of ${\mathcal P}$ is $d$
and if the origin of $\RR^{d}$ is the unique integer point belonging to the interior 
of ${\mathcal P}$.
A Fano polytope ${\mathcal P} \subset \RR^d$
is called {\em Gorenstein} if 
its dual polytope 
\[
\mathcal{P}^{\vee} = 
\{ \, {\bf x} \in \mathbb{R}^d \, | \, 
\langle {\bf x}, {\bf y}  \rangle \le 1 \mathrm{\ for\ all\ } {\bf y} \in \mathcal{P} \}
\] 
is integral, 
where $\langle {\bf x}, {\bf y} \rangle$ is a canonical inner product of $\RR^{d}$. 

We now come to an essential lemma on normal Fano polytopes.
We refer the reader to \cite{Sturm}
and \cite[Chapters $1$ and $5$]{dojo} 
for basic information on initial ideals of toric ideals, regular and unimodular
triangulations of integral convex polytopes.

\begin{Lemma}
\label{Gorenstein}
Let ${\mathcal P} \subset \RR^d$ be an integral convex polytope such that the origin is contained in its interior.
Let ${\mathcal P} \cap \ZZ^d = \{{\bf a}_1,\ldots, {\bf a}_n\}$
and 
\[
A = 
\begin{bmatrix} 
{\bf a}_1 & \cdots  & {\bf a}_n
\\
1 & \cdots & 1 
\end{bmatrix} \in \ZZ^{(d+1)\times n}.
\]
Suppose that 
$\ZZ A = \ZZ^{d+1}$ and that
there exists an ordering of the variables
$x_{i_{1}} < \cdots < x_{i_{n}}$
for which ${\bf a}_{i_1} = {\bf 0}$
such that the initial ideal ${\rm in}_<(I_A)$ of the toric ideal $I_{A}$
with respect to the reverse lexicographic order $<$ on $K[X]$ induced by 
the ordering is squarefree.
Then ${\mathcal P}$ is a normal Gorenstein Fano polytope. 
\end{Lemma}

\begin{proof}
The existence of an initial ideal as stated above guarantees that
${\mathcal P}$ possesses a unimodular triangulation $\Delta$ 
(i.e., a triangulation $\Delta$ of ${\mathcal P}$ 
such that the normalized volume of 
each maximal simplex in $\Delta$ is one) 
for which each maximal face of $\Delta$ contains the origin of $\RR^{d}$ as a vertex.
See \cite[Proposition 8.6 and Corollary 8.9]{Sturm}.
It then follows easily that
${\mathcal P}$ is Fano and that
the supporting hyperplane of
each facet of ${\mathcal P}$ is of the form
\[
\{
[z_1,\ldots,z_d]^t \in \RR^d
\ | \ 
a_{1} z_{1} + \cdots + a_{d} z_{d} = 1
\}
\]
with each $a_{i} \in \ZZ$.  Hence the dual polytope of ${\mathcal P}$ 
is integral.  
Furthermore, in general, the existence of a unimodular triangulation
of ${\mathcal P}$ says that ${\mathcal P}$ is normal
(see \cite[Proposition 13.15]{Sturm}).
\end{proof}

\medskip

\noindent
{\bf c) Centrally symmetric configuration}

In \cite{central}, 
the {\em centrally symmetric configuration} $A^\pm$ of a matrix $A
\in \mathbb{Z}^{d \times n}$
is introduced:
\[
A^{\pm} = 
\left[
\begin{array}{c|ccc|ccc}
0      &   &   &    & &   & \\
\vdots &   & A &    & &-A & \\
0      &   &   &    & &   & \\
\hline
1      & 1 & \cdots & 1 & 1 & \cdots & 1
\end{array}
\right].
\]
A basic fact obtained in \cite{central} is:

\begin{Proposition}[{\cite{central}}]
Let $A \in \mathbb{Z}^{d \times n}$ be a matrix
such that $\ZZ A = \ZZ^d$
and whose nonzero maximal minors are $\pm 1$. 
Then
the toric ideal $I_{A^{\pm}}$ of $A^{\pm}$ 
possesses a squarefree quadratic initial ideal with respect to 
a reverse lexicographic order for which the variable
corresponding to the column $[0, \ldots, 0, 1]^{t}$ is smallest.
In particular, ${\mathcal Q}^{\rm (sym)}_{A}$ is a 
normal Gorenstein Fano polytope.

\end{Proposition}

\begin{Remark}
In general, $A \in \ZZ^{d \times n}$ is called 
{\em unimodular} if ${\rm rank} (A)=d$ and all nonzero maximal minors of $A$ have the same absolute value.
If we assume $\ZZ A = \ZZ^d$, then 
$A$ is unimodular if and only if 
all nonzero maximal minors of $A$ are $\pm 1$. 
\end{Remark}

\begin{Example}
Let $A$ be the following configuration:
$$
A = 
\left[
\begin{array}{c|ccc}
0      &   &   &    \\
\vdots &   & A' &     \\
0      &   &   &    \\
\hline
1      & 1 & \cdots & 1
\end{array}
\right]
\mbox{ where }
A^{'}=
\begin{bmatrix}
1 &0 &0 &0 &0 &1 &1 &0 &0 &0  \\
1 &1 &0 &0 &0 &0 &0 &1 &0 &1  \\
0 &1 &1 &0 &0 &0 &1 &0 &1 &0  \\
0 &0 &1 &1 &0 &0 &0 &1 &0 &0  \\
0 &0 &0 &1 &1 &0 &0 &0 &1 &0  \\
1 &1 &1 &1 &1 &1 &1 &1 &1 &1
\end{bmatrix}.
$$
Then we have $\ZZ A =\ZZ^7$ but $A$ is not unimodular.
The initial ideal ${\rm in}_<(I_A)$ of $I_A$ 
with respect to any reverse lexicographic order $<$ is squarefree.
Moreover, $I_A$ has a quadratic Gr\"obner basis with respect to 
some reverse lexicographic order such that the smallest variable is $x_1$.
In addition, $K[A]$ is Gorenstein.
On the other hand,
${\mathcal Q}^{\rm (sym)}_{A}$ is not normal.
We can check that $I_{A^{\pm}}$ is not generated by quadratic binomials and ${\mathcal Q}^{\rm (sym)}_{A}$ is not Gorenstein
 by using the software package \texttt{CoCoA}
\cite{CoCoA}.
\end{Example}

\section{Centrally symmetric configurations of order polytopes}

In this section, we prove
that, for any poset $P$,
\begin{itemize}
\item
the toric ideal $I_{A_{P}^{\pm}}$ possesses
a squarefree quadratic initial ideal
with respect to a reverse lexicographic order 
such that the smallest variable corresponds to the origin;
(Theorem \ref{Grobner});
\item
${\mathcal Q}^{\rm (sym)}_{A_P} $ is a normal Gorenstein Fano polytope (Corollary \ref{NGF}). 
\end{itemize}


Let $P$ be a poset on $[d] = \{1, \ldots, d\}$ and let $K[s, T^{\pm 1}] = K[s, t_1, t_1^{-1}, \ldots, t_d, t_d^{-1}]$ be a Laurent polynomial ring in $d+1$ variables over $K$.
%
We set $S_P$ to be the polynomial ring 
in variables
$z$, $x_I$ ($\emptyset \neq I \in \mathcal{J}(P)$), 
and $y_I$  ($\emptyset \neq I \in \mathcal{J}(P)$)
over $K$. 
We define the ring homomorphism 
$\pi : S_P \to K[s, T^{\pm 1}]$ 
by setting $\pi(z) = s$, $\pi(x_I) = s \prod_{i \in I} t_i$ and  $\pi(y_I) = s \prod_{i \in I} t_i^{-1}$. 
Then the toric ideal $I_{A_{P}^{\pm}}$ is the kernel of $\pi$
and the toric ring $K[ A_{P}^{\pm} ]$ is the image of $\pi$. 
Let $<$ be a reverse lexicographic order on $S_P$ which satisfies 
$x_I < x_J$ and $y_I < y_J$ for all $I, J \in \mathcal{J}(P)$
with $I \subset J$. 
Here we set $x_\emptyset = y_\emptyset = z$
and hence $z$ is the smallest variable in $S_P$.

\begin{Example}\label{example}
Let $P$ be the poset on $\{1, 2, 3, 4, 5\}$ with the partial order $1 < 3, 2 < 3, 2 < 4$ and $4 < 5$. 
In this case, $A_P$ is the following matrix

$$
\begin{bmatrix}
1 &0 &1 &0 &1 &1 &0 &1 &1 &1  \\
0 &1 &1 &1 &1 &1 &1 &1 &1 &1  \\
0 &0 &0 &0 &1 &0 &0 &1 &0 &1  \\
0 &0 &0 &1 &0 &1 &1 &1 &1 &1  \\
0 &0 &0 &0 &0 &0 &1 &0 &1 &1  
\end{bmatrix}.
$$
Then the initial ideal of $I_{A_P^{\pm}}$ with respect to a reverse lexicographic order $<$ is generated by 66 squarefree quadratic monomials. 
\end{Example}

\begin{Theorem}\label{Grobner}
Let $\mathcal{G}$ be the set of binomials of the following types: 
\begin{eqnarray*}
& & x_I x_J - x_{I \cup J} x_{I \cap J}, \ \ 
y_I y_J - y_{I \cup J} y_{I \cap J} \hspace{5mm} 
(I \not\sim J) \\
& & 
x_I y_J - x_{I \setminus k} y_{J \setminus k} \hspace{5mm} 
(k \mbox{ is a maximal element of both } I \mbox{ and } J).
\end{eqnarray*} 
Then $\mathcal{G}$ is a Gr\"{o}bner bases of $I_{A_{P}^{\pm}}$ with respect to a reverse lexicographic order $<$,
and the initial monomial of each binomial in ${\mathcal G}$ is the first monomial. 
\end{Theorem}

\begin{proof}
Let $\mathrm{in}(\mathcal{G}) = \langle \mathrm{in}_<(g) \mid g \in \mathcal{G} \rangle$. 
Assume that $\mathcal{G}$ is not a Gr\"{o}bner bases of $I_{A_{P}^{\pm}}$. 
Then there exists a non-zero irreducible homogeneous binomial $f = u - v \in I_{A_{P}^{\pm}}$ such that 
neither $u$ nor $v$ belongs to $\mathrm{in}(\mathcal{G})$. 
It then follows that there exist 
$\mathcal{I}$, $\mathcal{J}$, $\mathcal{I}'$, and $\mathcal{J}'$
satisfying:
\[
u = z^{\alpha} \prod_{I \in \mathcal{I}}x_{I}^{p_I} \prod_{J \in \mathcal{J}}y_{J}^{q_J}, \ \ 
v = z^{\alpha '} \prod_{I' \in \mathcal{I'}}x_{I'}^{p'_{I'}} \prod_{J' \in \mathcal{J'}}y_{J'}^{q'_{J'}}, 
\]
where $0< p_I, q_J, p'_{I'}, q'_{J'} \in \mathbb{Z}$ for all 
$I \in \mathcal{I}, J \in \mathcal{J}$, 
$I' \in \mathcal{I}'$, and $J' \in \mathcal{J}'$. 

Since $x_I y_I \nmid u, v$ for all $\emptyset \neq I \in \mathcal{J}(P)$, we have $\mathcal{I} \cap \mathcal{J} = \mathcal{I'} \cap \mathcal{J'} = \emptyset$. 
Moreover, since $x_I x_J, y_I y_J \nmid u, v$ for all $I \not\sim J$, we can see that $\mathcal{I, J, I'}$, and $\mathcal{J'}$ are totally ordered sets. 
In addition, we have $\mathcal{I} \cap \mathcal{I'} = \mathcal{J} \cap \mathcal{J'} = \emptyset$ since $f$ is irreducible. 

Let $p = \sum_{I \in \mathcal{I}} p_I$, $q = \sum_{J \in \mathcal{J}} q_J$, $p' = \sum_{I' \in \mathcal{I'}} p'_{I'}$ and 
$q' = \sum_{J' \in \mathcal{J'}} q'_{J'}$. 
Then 
\[
\pi(u) = s^{\alpha + p + q} \prod_{I \in \mathcal{I}} \left( \prod_{i \in I} t_i \right)^{p_I} \prod_{J \in \mathcal{J}} \left( \prod_{j \in J} t^{-1}_j \right)^{q_{J}}, 
\]
\[
\pi(v) = s^{\alpha' + p' + q'} \prod_{I' \in \mathcal{I'}} \left( \prod_{i \in I'} t_i \right)^{p'_{I'}} \prod_{J' \in \mathcal{J'}} \left( \prod_{j \in J'} t^{-1}_j \right)^{q'_{J'}}. 
\]
Since $\pi(u) = \pi(v)$, 
\begin{equation}
\sum_{\substack{I \in \mathcal{I} \\ a \in I}} p_{I} - \sum_{\substack{J \in \mathcal{J} \\ a \in J}} q_{J} = \sum_{\substack{I' \in \mathcal{I}' \\ a \in I'}} p'_{I'} -  \sum_{\substack{J' \in \mathcal{J}' \\ a \in J'}} q'_{J'} 
\end{equation}
holds for all $a \in P$. 
In other words,
\begin{equation}
\sum_{\substack{I \in \mathcal{I} \\ a \in I}} p_{I} 
+
\sum_{\substack{J' \in \mathcal{J}' \\ a \in J'}} q'_{J'} 
= \sum_{\substack{I' \in \mathcal{I}' \\ a \in I'}} p'_{I'} 
+
\sum_{\substack{J \in \mathcal{J} \\ a \in J}} q_{J} 
\end{equation}
holds for all $a \in P$.

Let $I_{\max}$, $I'_{\max}$, $J_{\max}$ and $J'_{\max}$ be the maximal elements of $\mathcal{I}$, $\mathcal{I}'$, 
$\mathcal{J}$ and $\mathcal{J}'$, respectively.
(If $\mathcal{J} =\emptyset$, then let $J_{\max}= \emptyset$.)
Since $x_{I_{\max}} y_{J_{\max}}$ does not belong to ${\rm in}({\mathcal G})$, $I_{\max}$ and $J_{\max}$ do not have a common maximal element.
Note that the left (resp.~right) side of equation (2)
is not zero if and only if
$a \in P$ belongs to $I_{\max} \cup J'_{\max}$
(resp.~$I'_{\max} \cup J_{\max}$).
Hence $I_{\max} \cup J'_{\max} = I'_{\max} \cup J_{\max}$ holds.  Suppose that $I'_{\max} = \emptyset$.
Then $I_{\max} \cup J'_{\max} = J_{\max}$. 
If $J_{\max} = \emptyset$, then we have 
$I_{\max} =I'_{\max} = J_{\max} = J'_{\max} = \emptyset$.
Therefore $f = 0$, that is, $u = v$.  
This is a contradiction. 
If $J_{\max} \neq \emptyset$, then
let $\{b_1, \ldots, b_s\}$ be the set of maximal elements of $J_{\max}$.  
Then $J_{\max} = \bigcup_{i = 1}^{s} \langle b_i \rangle$, where $\langle b_i \rangle = \{ p \in P \mid p \le b_i \}$. 
Since $I_{\max}$ and $J_{\max}$ do not share a common maximal element, $b_i \not\in I_{\max}$ 
for all $1 \le i \le s$. 
Then $b_i \in J'_{\max}$ and we have $J_{\max} \subset J'_{\max}$. 
Hence $J'_{\max} \subset I_{\max} \cup J'_{\max} = J_{\max} \subset J'_{\max}$. 
Therefore we have $J_{\max} = J'_{\max}$, but this contradicts that $\mathcal{J} \cap \mathcal{J}' = \emptyset$. 
Hence $I'_{\max} \neq \emptyset$. 
Similarly, we have $I_{\max} \neq \emptyset$. 

Since $\mathcal{I} \cap \mathcal{I}' = \emptyset$,
%
we have $I_{\max} \neq I'_{\max}$. 
Hence we may assume that $I_{\max} \not\subset I'_{\max}$. 
Then there exists $a_{\max} \in I_{\max}$ such that $a_{\max}$ is a maximal element of $I_{\max}$ and 
$a_{\max} \not\in I'_{\max}$. 
Then $a_{\max} \in J_{\max}$. 
Since $a_{\max}$ is not a maximal element of $J_{\max}$, we have 
$\{ c \in J_{\max} \mid c > a_{\max}\}  \neq \emptyset $. 
Let 
$\{ c_1, \ldots, c_t \} = \{ c \in J_{\max} \mid c > a_{\max}\} $. 

For each $1 \le j \le t$, $c_j \not\in I_{\max}$ since $c_j > a_{\max}$. 
Thus we have $c_j \not\in I$ for all $I \in \mathcal{I} \cup \mathcal{I}'$ and $a_{\max} \not\in I'$ for all $I' \in \mathcal{I}'$. 
Hence we have
\[
- \sum_{\substack{J \in \mathcal{J} \\ c_j \in J}} q_{J} =   - \sum_{\substack{J' \in \mathcal{J}' \\ c_j \in J'}} q'_{J'}\ , \hspace{5mm} \sum_{\substack{I \in \mathcal{I} \\ a_{\max} \in I}}p_{I} - \sum_{\substack{J \in \mathcal{J} \\ a_{\max} \in J}} q_{J} = \ - \sum_{\substack{J' \in \mathcal{J}' \\ a_{\max} \in J'}} q'_{J'} 
\]
by equation (1). 
Note that if $c_j \in J \in {\mathcal J}(P)$, then $a_{\max} \in J$.
Thus we have
\[
\sum_{\substack{I \in \mathcal{I} \\ a_{\max} \in I}} p_{I}\ - \sum_{\substack{J \in \mathcal{J} \\ a_{\max} \in J, c_j \not\in J}} q_{J} \ \ = \ \ 
-\sum_{\substack{J' \in \mathcal{J}' \\ a_{\max} \in J', c_j \not\in J'}} q'_{J'} \le 0. 
\]
Since
$$
\sum_{\substack{I \in \mathcal{I} \\ a_{\max} \in I}} p_{I}
\ge 
p_{I_{\max}}
>0,
$$
there exists $J_j \in \mathcal{J}$ such that $a_{\max} \in J_j$ and $c_j \not\in J_j$
 for each $1 \le j \le t$. 

We may assume that $J_1 \subset \cdots \subset J_t$ 
$(\subset J_{\max})$
since $\mathcal{J}$ is a totally ordered set. 
Then $J_1 \cap \{c_1,\ldots, c_t\} = \emptyset$.
Thus $a_{\max}$ is a maximal element of 
both $I_{\max}$ and $J_1$, 
and hence $x_{I_{\max}} y_{J_1}$ belongs to ${\rm in} ({\mathcal G})$.
This is a contradiction. 
Therefore $\mathcal{G}$ is a Gr\"{o}bner basis of $I_{A_{P}^{\pm}}$.  
\end{proof}

In Theorem \ref{Grobner}, the initial ideal is squarefree and quadratic. 
By Lemma \ref{Gorenstein}, we have the following:


\begin{Corollary}
\label{NGF}
Let $P$ be a poset.
Then ${\mathcal Q}^{\rm (sym)}_{A_P}$ is a normal Gorenstein Fano polytope.
\end{Corollary}

Note that the matrix $A_P$ of a poset $P$ is not necessarily unimodular.
A characterization of the unimodularity of $A_P$ 
appeared in \cite[Example 3.6]{OHeHi} without proof.
Here, we recall the characterization with a proof.
%
Let $\mathbb{N}^2$ be the distributive lattice consisting of all pairs $(i, j)$ of nonnegative integers 
with the partial order $(i, j) \le (k, \ell)$ if and only if $i \le k$ and $j \le \ell$.  
A distributive lattice $D$ is said to be {\em planar} (see \cite[p. 436]{Lattice}) if $D$ is a finite sublattice of $\mathbb{N}^2$ with $(0, 0) \in D$ 
which satisfies the following: for any $(i, j),\ (k, \ell) \in D$ with $(i, j) < (k, \ell)$, there exists a chain (totally ordered subset) of $D$ 
of the form $(i, j) = (i_0, j_0) < (i_1, j_1) < \cdots < (i_s, j_s) = (k, \ell)$ such that $i_{k + 1} + j_{k + 1} = i_k + j_k + 1$ for all $k$. 

\begin{Proposition}[\cite{OHeHi}]
\label{unimodular}
Let ${\mathcal J}(P)$ be the distributive lattice associated with a poset $P$. 
Then $A_P$ is unimodular if and only if ${\mathcal J}(P)$ is planar. 
\end{Proposition}
\begin{proof}
Assume that ${\mathcal J}(P)$ is planar. 
Then $K[A_P]$ is isomorphic to $K[A]$ where $A$ is 
the vertex-edge incidence matrix of a bipartite graph
(see \cite{KBG}). 
Since $A$ is unimodular \cite{Sch}, $A_P$ is also unimodular. 

Conversely, assume that ${\mathcal J}(P)$ is not planar. 
Then there exist $a, b, c \in P$ such that $a \not\sim b$, $a \not\sim c$ and $b \not\sim c$. 
Let $I = \langle a \rangle \cup \langle b \rangle \cup \langle c \rangle$. 
Then the sublattice
$
\left\{
J \in {\mathcal J}(P)
\ \left| \ 
 I \setminus \{a, b, c\} \subset J \subset I
\right.
\right\}
$
 of ${\mathcal J}(P)$
is isomorphic to ${\mathcal J} (P')$,
where $P'=\{a,b,c\}$ is an antichain.
It is easy to see that 
$$A_{P'}=
\begin{bmatrix}
1 & 0 & 0 & 1 & 1 & 0 & 1\\
0 & 1 & 0 & 1 & 0 & 1 & 1\\
0 & 0 & 1 & 0 & 1 & 1 & 1
\end{bmatrix}
$$ is not unimodular.
Since $K[A_{P'}]$ is a combinatorial pure subring (see \cite{OHeHi}) of $K[A_P]$, $A_P$ is not unimodular. 
 \end{proof}

\section{The $\delta$-vector of ${\mathcal Q}^{\rm (sym)}_{A_P}$ of an antichain poset $P$}

In this section, we show that
the $\delta$-vector of ${\mathcal Q}^{\rm (sym)}_{A_P}$
of an antichain poset $P$ is
the Eulerian number (see \cite{sbook}).
%

%
First, we review Ehrhart theory on integral convex polytopes.
Let ${\mathcal P} \subset \RR^d$  be an integral convex polytope of dimension $d$.
Given integers $t=1,2,\ldots$, we set
$i({\mathcal P},t) = \# (t {\mathcal P} \cap \ZZ^d)$.
Then $i({\mathcal P},t) $ is a polynomial in $t$
of degree $d$ and is called the {\em Ehrhart polynomial}
of ${\mathcal P}$.
Its generating function satisfies
$$
1 + \sum_{t=1}^{\infty} i({\mathcal P},t) 
\lambda^t
=
\frac{\delta_{\mathcal P}(\lambda)}{(1-\lambda)^{d+1}}
$$
where $\delta_{\mathcal P}(\lambda) = \sum_{i=0}^d \delta_i \lambda^i$ is a polynomial 
in $\lambda$ of degree $\le d$.
The vector $(\delta_0,\ldots, \delta_d)$ is called the {\em $\delta$-vector} of ${\mathcal P}$.
It is known that a Fano polytope
${\mathcal P} \subset \RR^d$ is Gorenstein if and only if 
$\delta_i = \delta_{d-i}$ for all $0 \leq i \leq d$.
See \cite{hibibook}.

Let $P$ be an antichain poset on $[d]$.
One of the referees pointed out that 
${\mathcal Q}^{\rm (sym)}_{A_P}$
is the Minkowski sum $C_d + L$ where
$C_d$ is the unit $d$-cube in $\RR^d$
and $L$ is the closed line segment
whose end points are the origin and
the vector ${\bf v} = [-1, \ldots, -1]^t \in \RR^d$.
In order to see this fact,
note that
${\mathcal Q}^{\rm (sym)}_{A_P}$
is the convex hull of $C_d \cup (-C_d)$
and that $-C_d$ is the Minkowski sum $C_d + {\bf v}$.
Hence ${\mathcal Q}^{\rm (sym)}_{A_P}$ is a {\it zonotope}, that is, a Minkowski sum of
closed line segments:
$$
{\mathcal Q}^{\rm (sym)}_{A_P}
=
\{
r_1 {\bf e}_1 + \cdots + r_d {\bf e}_d + r_{d+1} {\bf v}
\ | \ 
0 \leq r_i \leq 1\  (i = 1,2, \ldots, d+1)
\}
.$$
By using \cite[Theorem 2.2]{StanleyZono},
we have the following:

\begin{Proposition}
\label{antichain}
Let $P$ be an antichain poset on $[d]$.
Then we have
\begin{eqnarray*}
i( {\mathcal Q}^{\rm (sym)}_{A_P} ,t)
&=&
(t + 1)^{d + 1} - t^{d + 1},\\
1+
\sum_{t=1}^\infty
i( {\mathcal Q}^{\rm (sym)}_{A_P} ,t) \lambda^t
&=&
\frac{\sum_{i = 0}^{d} A(d + 1, i)\lambda^i}{(1 - \lambda)^{d + 1}}, 
\end{eqnarray*}
where $A(d + 1, i)$ is the Eulerian number. 
\end{Proposition}

\begin{proof}
Let $B = \{{\bf e}_1, \ldots, {\bf e}_d, {\bf v}\}$.
Then $B$ is not linearly independent and
any proper subset of $B$ is linearly independent.
It is easy to see that any nonzero minor of
the matrix 
$[{\bf e}_1, \ldots, {\bf e}_d, {\bf v}] \in \ZZ^{d \times (d+1)}$
is $\pm 1$.
Hence by \cite[Theorem 2.2]{StanleyZono},
we have 
$$i( {\mathcal Q}^{\rm (sym)}_{A_P} ,t)
=
\sum_{X \subsetneq B}
t^{\sharp X}
=(t + 1)^{d + 1} - t^{d + 1}
.$$
By a well-known identity
\[
\sum_{t = 1}^{\infty} t^d \lambda^t = \frac{\sum_{i = 0}^{d} A(d, i) \lambda^{i + 1}}{(1 - \lambda)^{d + 1}}
\]
for the Eulerian number,
it is easy to show 
$$1+
\sum_{t=1}^\infty
((t + 1)^{d + 1} - t^{d + 1})\lambda^t
=
\frac{\sum_{i = 0}^{d} A(d + 1, i)\lambda^i}{(1 - \lambda)^{d + 1}}.
$$
\end{proof}

\begin{Remark}
Let $P$ be an antichain poset on $[d]$.
It is known that 
$$
1+
\sum_{t=1}^\infty
i({\mathcal O}(P),t) \lambda^t
=
\frac{\sum_{i = 0}^{d} A(d, i) \lambda^i}{(1 - \lambda)^{d + 1}}
.$$ 
\end{Remark}

\begin{Example}
Let $P$ be the poset as appeared in Example \ref{example}. 
Then the $\delta$-vector of $\mathcal{Q}_{A_P}^{(\mathrm{sym})}$ is $(1, 15, 54, 54, 15, 1)$. 
Note that $\mathcal{O}(P)$ is not Gorenstein since $P$ is not pure, and its $\delta$-vector is $(1, 5, 3)$. 
\end{Example}

\bigskip

\noindent
{\bf Acknowledgment.}
The authors are grateful to anonymous referees for
their careful reading, useful suggestions, and helpful comments.

\end{document}